\documentclass[fleq]{amsart}
\usepackage{amssymb,amsmath,latexsym,amsbsy,color,enumerate}
\numberwithin{equation}{section}

\newtheorem{theorem}{Theorem}

\newtheorem{lemma}[theorem]{Lemma}

\newtheorem{proposition}[theorem]{Proposition}
\newtheorem{corollary}[theorem]{Corollary}

\begin{document}
 \title{The simplicity of the first spectral radius of a meromorphic map}
 \author{Tuyen Trung Truong}
    \address{Department of mathematics, Syracuse University, Syracuse NY 13244, USA}
 \email{tutruong@syr.edu}
\thanks{}
    \date{\today}
    \begin{abstract}Let $X$ be a compact K\"ahler manifold and let $f:X\rightarrow X$ be a dominant rational map which is $1$-stable. Let $\lambda _1$ and $\lambda _2$ be the first and second dynamical degrees of $f$. If $\lambda _1^2>\lambda _2$, then we show that $\lambda _1$ is a simple eigenvalue of $f^*:H^{1,1}(X)\rightarrow H^{1,1}(X)$, and moreover the unique eigenvalue of modulus $>\sqrt{\lambda _2}$. A variant of the result, where we consider the first spectral radius in the case the map $f$ may not be $1$-stable, is also given. An application is stated for bimeromorphic selfmaps of $3$-folds.  
    
In the last section of the paper, we prove analogs of the above results in the algebraic setting, where $X$ is a projective manifold over an algebraic closed field of characteristic zero, and $f:X\rightarrow X$ is a rational map. Part of the section is devoted to defining dynamical degrees in the algebraic setting. We stress that here the dynamical degrees of rational maps can be defined over any algebraic closed field, not necessarily of characteristic zero.      
\end{abstract}
\maketitle
\section{Introduction}
 Let $X$ be a compact K\"ahler manifold of dimension $k$ with a K\"ahler form $\omega _X$, and let $f:X\rightarrow X$ be a dominant meromorphic map. For $0\leq p\leq k$, the $p$-th dynamical degree $\lambda _p(f)$ of $f$ is defined as follows 
\begin{eqnarray*} 
\lambda _p(f)=\lim _{n\rightarrow\infty}(\int _X(f^n)^*(\omega _X^p)\wedge \omega _X^{k-p})^{1/n}=\lim _{n\rightarrow\infty}r_p(f^n)^{1/n},
\end{eqnarray*}
where $r_p(f^n)$ is the spectral radius of the linear map $(f^n)^*:H^{p,p}(X)\rightarrow H^{p,p}(X)$ (see Russakovskii-Shiffman \cite{russakovskii-shiffman} for the case where $X=\mathbb{P}^k$, and Dinh-Sibony \cite{dinh-sibony10}\cite{dinh-sibony1} for the general case; see also Guedj \cite{guedj4} and Friedland \cite{friedland}). The dynamical degrees are log-concave, in particular $\lambda _1(f)^2\geq \lambda _2(f)$. In the case $f^*:H^{2,2}(X)\rightarrow H^{2,2}(X)$ preserves the cone of psef classes (i.e. those $(2,2)$ cohomology classes which can be represented by positive closed $(2,2)$ currents), then we have an analog $r_1(f)^2\geq r_2(f)$ (see Theorem \ref{TheoremTheCaseNotStable}).

The present paper concerns the first dynamical degree $\lambda _1(f)$ and more generally the first spectral radius $r_1(f)$. We will say that $f$ is $1$-stable if for any $n\in \mathbb{N}$, $(f^n)^*=(f^*)^n$ on $H^{1,1}(X)$ (the first use of this notion appeared in the paper Fornaess-Sibony \cite{fornaess-sibony1} in the case of rational selfmaps of projective spaces). When $f$ is $1$-stable, we have $\lambda _1(f)=r_1(f)$. The first main result of this paper is the following
\begin{theorem} Let $X$ be a compact K\"ahler manifold of dimension $k$, and let $f:X\rightarrow X$ be a dominant meromorphic map which is $1$-stable. Assume that $\lambda _1(f)^2>\lambda _2(f)$. Then $\lambda _1(f)$ is a simple eigenvalue of $f^*:H^{1,1}(X)\rightarrow H^{1,1}(X)$. Further, $\lambda _1(f)$ is the only eigenvalue of modulus greater than $\sqrt{\lambda _2(f)}$.  
\label{TheoremFirstDynamicalDegreeIsSimple}\end{theorem}
Theorem \ref{TheoremFirstDynamicalDegreeIsSimple} answers Question 3.3 in Guedj \cite{guedj3}. It was known when $f$ is holomorphic, see e.g. Cantat-Zeghib \cite{cantat-zeghib} where the case of holomorphic maps of $3$-folds is explicitly stated. An immediate consequence of Theorem \ref{TheoremFirstDynamicalDegreeIsSimple} is that if $f$ is $1$-stable and $\lambda _1(f)^2>\lambda _2(f)$, then the "degree growth" of $f$ satisfies $deg(f^n)=c\lambda _1(f)^n+O(\tau ^n)$ for some constants $c>0$ and $\tau <\lambda _1(f)$. In the case $X$ is a surface, the same estimate for the degree growth  was obtained in Boucksom-Favre-Jonsson \cite{boucksom-favre-jonsson} where the condition $f$ is $1$-stable is not needed. The conclusion of Theorem \ref{TheoremFirstDynamicalDegreeIsSimple} that $\lambda _1(f)$ is simple is very helpful in constructing Green currents and proving equi-distribution properties toward it (see e.g. Guedj \cite{guedj3}, Diller-Guedj \cite{diller-guedj} and Bayraktar \cite{bayraktar}).      

When $X$ is a compact K\"ahler surface, Diller-Favre \cite{diller-favre} proved a stronger conclusion than that of Theorem \ref{TheoremFirstDynamicalDegreeIsSimple} where the condition of $1$-stability is dropped. The following variant of Theorem \ref{TheoremFirstDynamicalDegreeIsSimple} gives a generalization of Diller and Favre's result to higher dimensions. Recall that $r_1(f)$ is the spectral radius of $f^*:H^{1,1}(X)\rightarrow H^{1,1}(X)$ and $r_2(f)$ is the spectral radius of $f^*:H^{2,2}(X)\rightarrow H^{2,2}(X)$.

\begin{theorem}
Let $X$ be a compact K\"ahler manifold, and let $f:X\rightarrow X$ be a dominant meromorphic map. Assume that $f^*:H^{2,2}(X)\rightarrow H^{2,2}(X)$ preserves the cone of psef classes. Then  

1) We have $r_1(f)^2\geq r_2(f)$.

2) Assume moreover that $r_1(f)^2>r_2(f)$. Then $r_1(f)$ is a simple eigenvalue of $f^*:H^{1,1}(X)\rightarrow H^{1,1}(X)$. Further, $r_1(f)$ is the only eigenvalue of modulus greater than $\sqrt{r_2(f)}$.  

\label{TheoremTheCaseNotStable}\end{theorem}
 
As a consequence, we obtain the following
\begin{corollary}
Let $X$ be a compact K\"ahler manifold of dimension $3$. Let $f:X\rightarrow X$ be a bimeromorphic map such that both $f$ and $f^{-1}$ are $1$-stable. Assume moreover that $\lambda _1(f)>1$. Then either $f$ or $f^{-1}$ satisfies the conclusions of Theorem \ref{TheoremFirstDynamicalDegreeIsSimple}.
\label{CorollaryBimeromorphicDimension3}\end{corollary}
\begin{proof}
Observe that $\lambda _1(f^{-1})=\lambda _2(f)$ and $\lambda _2(f^{-1})=\lambda _1(f)$. Hence when $\lambda _1(f)>1$, at least one of the following conditions hold: $\lambda _1(f)^2>\lambda _2(f)$ and $\lambda _1(f^{-1})^2>\lambda _2(f^{-1})$.
\end{proof}
Corollary \ref{CorollaryBimeromorphicDimension3} can be applied to pseudo-automorphisms $f:X\rightarrow X$ of a $3$-fold $X$ with $\lambda _1(f)>1$.  By definition (see e.g. \cite{dolgachev-ortland}), a bimeromorphic map $f:X\rightarrow X$ is pseudo-automorphic if there are subvarieties $V,W$ of codimension at least $2$ so that $f:X-V\rightarrow X-W$ is biholomorphic. If $X$ has dimension $3$, then any pseudo-automorphism $f:X\rightarrow X$ is both $1$-stable and $2$-stable (see Bedford-Kim \cite{bedford-kim}). The first examples of pseudo-automorphisms with first dynamical degree larger than $1$ on blowups of $\mathbb{P}^3$ were given in \cite{bedford-kim}, by studying linear fractional maps in dimension $3$. There are now several other examples in any dimension (see e.g.  Perroni-Zhang  \cite{perroni-zhang}, Blanc \cite{blanc} and Oguiso \cite{oguiso2}). 

The key tools in the proofs of Theorems \ref{TheoremFirstDynamicalDegreeIsSimple} and \ref{TheoremTheCaseNotStable} are the Hodge index theorem (Hodge-Riemann bilinear relations), Hironaka's elimination of indeterminacies for meromorphic maps, and a pull-push formula for blowups along smooth centers. Section 2 is devoted to the proofs of Theorems \ref{TheoremFirstDynamicalDegreeIsSimple} and \ref{TheoremTheCaseNotStable}.

All of the above results have analogs in the algebraic setting, where $X$ is a projective manifold over an algebraic closed field of characteristic zero, and $f:X\rightarrow X$ is a rational map. This will be done in Section 3 (see Theorems \ref{TheoremFirstDynamicalDegreeIsSimpleAlgebraicCase} and \ref{TheoremTheCaseNotStableAlgebraicCase}). We conclude this introduction noting some remarks. Unlike the case of compact K\"ahler manifolds, a priori there are no smooth forms, groups $H^{p,p}(X)$ and "regularization of currents" available in the algebraic case. In stead, we use the groups of algebraic cycles modulo numerical equivalence and Chow's moving lemma to define dynamical degrees. The analog of the Hodge index theorem is then the Grothendieck-Hodge index theorem. We stress that here the dynamical degrees can be defined for rational maps over any algebraic closed field, not necessarily of characteristic zero. 

{\bf Remark.} After this paper was written, the author was informed by Charles Favre of another algebraic method to define dynamical degrees.  

{\bf Acknowledgements.} We are benefited from many discussions and correspondences with Tien-Cuong Dinh, Charles Favre, Mattias Jonsson, J\'anos Koll\'ar, Pierre Milman, Viet-Anh Nguyen and Claire Voisin on various topics: Riemann-Zariski space, dynamical degrees, mixed Hodge-Riemann theorem, Grothendieck-Hodge index theorem, Hironaka's resolution of singularities  and Hironaka's elimination of indeterminacies. The author is grateful to Mattias Jonsson whose suggestion of extending Theorems \ref{TheoremFirstDynamicalDegreeIsSimple} and \ref{TheoremTheCaseNotStable} to the algebraic setting and whose help with an earlier version of the paper made the results and the presentation of the paper better. The author also would like to thank Dan Coman, Eric Bedford, Keiji Oguiso, H\'elene Esnault, and Turgay Bayraktar for their help and useful comments.   

\section{Proofs of Theorems \ref{TheoremFirstDynamicalDegreeIsSimple} and \ref{TheoremTheCaseNotStable}}
Let $X$ and $Y$ be compact K\"ahler manifolds and let $h:X\rightarrow Y$ be a dominant meromorphic map. By Hironaka's elimination of indeterminacies (see e.g. Corollary 1.76 in Koll\'ar \cite{kollar} and Theorem 7.21 in Harris \cite{harris} for the case $X$ is projective, and see Hironaka \cite{hironaka} and Moishezon \cite{moishezon} for the general case), there is a compact K\"ahler manifold $Z$, a map $\pi :Z\rightarrow X$ which is a finite sequence of blowups along smooth centers, and a surjective holomorphic map $g:Z\rightarrow Y$, so that $h=g\circ \pi ^{-1}$. (Since the analytic case of Hironaka's elimination of indeterminacies is less known, we give here a sketch of how to prove it, cf. the paper Ishii-Milman \cite{ishii-milman} for related ideas. We thank Pierre Milman for his generous help with this. Consider $\Gamma $ a resolution of singularities of the graph $\Gamma _h$, and let $p,\gamma :\Gamma \rightarrow X,Y$ be the induced holomorphic maps. In particular $p:\Gamma\rightarrow X$ is a modification. By global Hironaka's flattening theorem, we can find a finite sequence of blowups $\pi :X'\rightarrow X$ along smooth centers, and let $\pi _{\Gamma}:\Gamma '\rightarrow \Gamma $ be the corresponding blowup along the ideals which are pullbacks by $p$ of the ideals of the centers of the blowup $\pi$, so that the induced map $p':\Gamma '\rightarrow X'$ is still holomorphic, bimeromorphic and flat. A priori, $\Gamma '$ may be singular. But a holomorphic, bimeromorphic and flat map must actually be a biholomorphic map. Therefore, $\Gamma '$ is also smooth, $p'$ is biholomorphic, and the holomorphic maps $\pi :Z=X'\rightarrow X$ and $g=\gamma \circ \pi _{\Gamma}\circ p'^{-1} :~ Z=X'\rightarrow Y$ are what needed.)  

For our purpose here, it is important to study the blowups whose center is a smooth submanifold of codimension exactly $2$. We consider first the case of a single blowup. We use the conventions that if $W$ is a subvariety then $[W]$ denotes the current of integration along $W$, and if $T$ is a closed current then $\{T\}$ denotes its cohomology class (for the case $T=[W]$ where $W$ is a subvariety, we write $\{W\}$ instead of $\{[W]\}$ for convenience). For two cohomology classes $u$ and $v$, we denote by $u.v$ the cup product.

We have the following pull-push formulas for a single blowup (a more precise version of this for birational surface maps was given in \cite{diller-favre})
\begin{lemma} Let $X$ be a compact K\"ahler manifold of dimension $k$. Let $\pi :Z\rightarrow X$ be a blowup of $X$ along a smooth  submanifold $W=\pi (E)$ of codimension exactly $2$. Let $E$ be the exceptional divisor and let $L$ be a general fiber of $\pi$.  

i) There is a constant $c_{E}\geq 0$ so that
\begin{eqnarray*}
(\pi )_*(\{E\}.\{E\})=-c_{E}\{W\}.
\end{eqnarray*}

ii) If $\alpha$ is a closed smooth $(1,1)$ form with complex coefficients on $Z$ then 
\begin{eqnarray*}
\pi ^*(\pi )_*(\alpha )=\alpha +(\{\alpha \}.\{L\})[E].
\end{eqnarray*}

iii) If $\alpha$ is a closed smooth $(1,1)$ form with complex coefficients on $Z$ then
\begin{eqnarray*}
(\pi )_*(\alpha \wedge [E])=c_{E}(\{\alpha \}.\{L\})[W].
\end{eqnarray*}

iv) If $\alpha$ is a closed smooth $(1,1)$ form with complex coefficients on $Z$ then
\begin{eqnarray*}
(\pi )_*((\pi )^*(\pi )_*(\alpha )\wedge \overline{\alpha })-(\pi )_*(\alpha \wedge \overline{\alpha })=c_{E}|\{\alpha \}.\{L\}|^2[W].
\end{eqnarray*}
\label{LemmaPullPushFormulaForOneBlowup}\end{lemma}
Remarks: 

1) If $X$ is projective, then $c_E=1$ in the lemma (see Lemma \ref{LemmaPullPushFormulaForOneBlowupAlgebraicCase}). We thank Charles Favre for showing this to us. 

2) Lemma \ref{LemmaPullPushFormulaForOneBlowup} i), iii), iv) and v) are trivially true when the center of blowup $W=\pi _1(E)$ has codimension at least $3$. For example, then in i) we have $\pi _*(\{E\}.\{E\})=0$. In fact, by the same argument as in the proof of i) below, the cohomology class $\pi _*(\{E\}.\{E\})$ can be represented by a difference of two positive closed $(2,2)$ currents supported in $W=\pi (E)$. Since $W$ has codimension at least $3$, it follows that $\pi _*(\{E\}.\{E\})=0$.  

\begin{proof}

i) By Demailly's regularization for positive closed $(1,1)$ currents (see Demailly \cite{demailly}, and also Dinh-Sibony \cite{dinh-sibony1}), there are positive closed smooth $(1,1)$ forms $\alpha _n,\beta _n$ of bounded masses so that $\alpha _n-\beta _n$ weakly converges to the current of integration $[E]$. Let $\alpha$ and $\beta$ be any cluster points of the currents $\alpha _n\wedge [E]$ and $\beta _n\wedge [E]$, then $\alpha$ and $\beta$ are positive closed $(2,2)$ currents with support in $E$ and in cohomology $\{\alpha -\beta \}=\{E\}.\{E\}$. Therefore $\pi _*(\{E\}.\{E\})$ can be represented by the difference $\pi _*(\alpha )-\pi _*(\beta )$ of two positive closed $(2,2)$ currents $\pi _*(\alpha )$ and $\pi _*(\beta )$. Each of the latter has support in $W=\pi (E)$,  hence since $W$ has codimension exactly $2$, each of them must be a multiple of the current of integration $[W]$ by the support theorem for normal currents. We infer
\begin{eqnarray*}
\pi _*(\{E\}.\{E\})=-c_{E}\{W\},
\end{eqnarray*} 
for a constant $c_{E}$. It remains to show that $c_{E}\geq 0$. To this end, we let $\omega _X$ be a K\"ahler form on $X$. Then we get
\begin{eqnarray*}
\{E\}.\{E\}.\{\pi ^*(\omega _X^{k-2})\}=(\pi )_*(\{E\}.\{E\}).\{\omega _X^{k-2}\}=-c_{E}\{W\}.\{\omega _X^{k-2}\}.
\end{eqnarray*}
Since $\{W\}.\{\omega _X^{k-2}\}=\{[W]\wedge \omega _X^{k-2}\}$ is a positive number (equal the mass of $W$), to show that $c_{E}\geq 0$ it suffices to show that $\{E\}.\{E\}.\{\pi ^*(\omega _X^{k-2})\}\leq 0$. If we can show that $\{E\}.\{\pi ^*(\omega _X^{k-2})\}=a\{L\}$ for some constant $a\geq 0$ then $\{E\}.\{E\}.\{\pi ^*(\omega _X^{k-2})\}=a\{E\}.\{L\}=-a\leq 0$ as wanted. To this end, first we observe that $\{E\}.\{\pi ^*(\omega _X^{k-2})\}=a\{L\}$ for some constant $a$, because $H^{k-1,k-1}(Z)$ is generated by $\pi _1^*H^{k-1,k-1}(X)$ and $\{L\}$, and by the projection formula $(\pi )_*(\{E\}.\{\pi ^*(\omega _X^{k-2})\})=(\pi )_*(\{E\}).\{\omega _X^{k-2})\}=0$. The constant $a$ then must be non-negative because $\{E\}.\{\pi ^*(\omega _X^{k-2})\}=\{[E]\wedge \pi ^*(\omega _X^{k-2})\}$ is a psef class. 

ii) This is a standard result using $\{E\}.\{L\}=-1$ (see also iii) below).

iii) Since $(\pi )_*(\alpha \wedge [E])$ is a normal $(2,2)$ current with support in $W=\pi (E)$ which is a subvariety of codimension $2$ in $X$, by support theorem it follows that there is a constant $c$ such that $(\pi )_*(\alpha \wedge [E])=c[W]$. It is clear that $c$ depends only on the cohomology class of $(\pi )_*(\alpha \wedge [E])$. Since $H^{1,1}(Z)$ is generated by $\pi ^*(H^{1,1}(X))$ and $\{E\}$, we can write $\{\alpha \}=a\pi ^*(\beta )+b\{E\}$ where $\beta \in H^{1,1}(X)$. Then using i) and the projection formula we obtain
\begin{eqnarray*}
(\pi )_*\{\alpha \wedge [E]\}&=&(\pi )_*(\{\alpha \}.\{E\})=b(\pi )_*(\{E\}.\{E\})\\
&=&-bc_{E}\{\pi (E)\}.  
\end{eqnarray*}
Therefore $c=-bc_{E}$. The constant $-b$ can be computed as follows
\begin{eqnarray*}
\{\alpha \}.\{L\}=(a\pi ^*(\beta )+b\{E\}).\{L\}=b\{E\}.\{L\}=-b.
\end{eqnarray*}
Hence $c=(\{\alpha\}.\{L\})c_{E}$ as claimed.

iv) We have
\begin{eqnarray*}
(\pi )_*(\pi ^*(\pi )_*(\alpha )\wedge \overline{\alpha})&=&(\pi )_*((\alpha +(\{\alpha \}.\{L\})[E])\wedge \overline{\alpha})\\
&=&(\pi )_*(\alpha \wedge \overline{\alpha} )+(\{\alpha \}.\{L\})(\pi )_*([E]\wedge \overline{\alpha})\\
&=&(\pi )_*(\alpha \wedge \overline{\alpha} )+c_{E}|\{\alpha \}.\{L\}|^2[\pi (E)].
\end{eqnarray*}
Thus iv) is proved.
\end{proof}

In particular, Lemma \ref{LemmaPullPushFormulaForOneBlowup} shows that for a single blowup $\pi :Z\rightarrow X$, if $\alpha$ is a closed smooth $(1,1)$ form with complex coefficients then $(\pi )_*((\pi )^*(\pi )_*(\alpha )\wedge \overline{\alpha })-(\pi )_*(\alpha \wedge \overline{\alpha })$ is a positive closed $(2,2)$ current. (If the center of blowup $W$ has codimension exactly $2$ then this follows from Lemma \ref{LemmaPullPushFormulaForOneBlowup} iv), while if $W$ has codimension at least $3$ then $(\pi )_*((\pi )^*(\pi )_*(\alpha )\wedge \overline{\alpha })-(\pi )_*(\alpha \wedge \overline{\alpha })=0$ as observed in the remarks after the statement of Lemma \ref{LemmaPullPushFormulaForOneBlowup}.) It follows that if $u\in H^{1,1}(Z)$ is a cohomology class with complex coefficients, then     
$\pi _*(u).\pi _*(\overline{u})-\pi _*(u.\overline{u})$ is a psef class, that is can be represented by a positive closed $(2,2)$ current. In fact, let $\alpha$ be a closed smooth $(1,1)$ form representing $u$. Then, $(\pi )_*(u.\overline{u})$ is represented by $(\pi )_*(\alpha \wedge \overline{\alpha})$, and by the projection formula $(\pi )_*(u).(\pi )_*(\overline{u})$ is represented by $(\pi )_*(\pi ^*(\pi )_*(\alpha )\wedge \overline{\alpha})$. Hence from iv), we infer that $\pi _*(u).\pi _*(\overline{u})-\pi _*(u.\overline{u})$ is psef, as claimed. We now give a generalization of this to the case of a finite blowup and to meromorphic maps. 

\begin{proposition}

1) Let $X$ be a compact K\"ahler manifold, and $\pi :Z\rightarrow X$ a finite composition of blowups along smooth centers. Further, let $u\in H^{1,1}(Z)$ be a $(1,1)$ cohomology class with complex coefficients. Then $(\pi )_*(u).(\pi )_*( \overline{u })-(\pi )_*(u . \overline{u})$ is a psef class.

2) Let $X$ and $Y$ be compact K\"ahler manifolds, and $h:X\rightarrow Y$ a dominant meromorphic map. Further, let $u\in H^{1,1}(Y)$ be a cohomology class with complex coefficients on $Y$. Then $h^*(u).h^*(\overline{u} )-h^*(u . \overline{u})$ is a psef class in $H^{2,2}(X)$.
\label{LemmaPullPushFormulaForMeromorphicMaps}\end{proposition}
\begin{proof}

1) We prove by induction on the number of single blowups performed. If $\pi$ is a single blowup then this follows from the above observation. Now assume that 1) is true when the number of single blowups performed is $\leq n$. We prove that 1) is true also when then number of single blowups performed is $\leq n+1$. We can decompose $\pi =\pi _1\circ \pi _2:Z\rightarrow Y\rightarrow X$, where $\pi _2:Z\rightarrow Y$ is a single blowup, and $\pi _1:Y\rightarrow X$ is a composition of $n$ single blowups. Apply the inductional assumption to $\pi _1$ and the cohomology class $(\pi _2)_*(u)$, we get
\begin{eqnarray*}
\pi _*(u).\pi _*(\overline{u})=(\pi _1)_*((\pi _2)_*(u)).(\pi _1)_*((\pi _2)_*(\overline{u}))\geq (\pi _1)_*((\pi _2)_*(u).(\pi _2)_*(\overline{u})).
\end{eqnarray*} 
Here the $\geq $ means that the difference of the two currents is psef. Now using the result for the single blowup $\pi _2$ and the fact that push-forward by the holomorphic map $\pi _1$ preserves psef classes, we have 
\begin{eqnarray*}
(\pi _1)_*((\pi _2)_*(u).(\pi _2)_*(\overline{u}))\geq (\pi _1)_*(\pi _2)_*(u.\overline{u})=\pi _*(u.\overline{u}).
\end{eqnarray*}
Hence $\pi _*(u).\pi _*(\overline{u})\geq \pi _*(u.\overline{u})$ as wanted.

2) By Hironaka's elimination of indeterminacies (see Hironaka \cite{hironaka} and Moishezon \cite{moishezon}), we can find a compact K\"ahler manifold $Z$, a finite blowup along smooth centers $\pi :Z\rightarrow X$ and a surjective holomorphic map $g:Z\rightarrow Y$ so that $h=g\circ \pi ^{-1}$. By definition $h^*(u)=\pi _*g^*(u)$ and   $h^*(u.\overline{u})$ $=$ $\pi _*(g^*(u.\overline{u}))$ $=$ $\pi _*(g^*(u).g^*(\overline{u}))$ (to see these equalities, we choose a smooth closed $(1,1)$ form $\alpha$ representing $u$ and see immediately the equalities on the level of currents). Therefore, apply 1) to the blowup $\pi :Z\rightarrow X$ and to the $(1,1)$ cohomology class $g^*(u)$ on $Z$, we obtain
\begin{eqnarray*}
h^*(u). h^*(\overline{u})-h^*(u .\overline{u})=\pi _*(g^*(u)). \pi _*(\overline{g^*(u )})-\pi _*(g^*(u ). \overline{g^*(u )})\geq 0.
\end{eqnarray*}
\end{proof}

For the proofs of Theorems \ref{TheoremFirstDynamicalDegreeIsSimple} and \ref{TheoremTheCaseNotStable} we need to use the famous Hodge index theorem (Hodge-Riemann bilinear relations, see e.g. the last part of Chapter 0 in Griffiths-Harris \cite{griffiths-harris}). Let $X$ be a compact K\"ahler manifold of dimension $k$. Let $w\in H^{1,1}(X)$ be the cohomology class of a K\"ahler form on $X$. We define a Hermitian quadratic form which for cohomology classes with complex coefficients $u,v\in H^{1,1}(X)$ takes the value
\begin{eqnarray*}
\mathcal{H}(u,v)=u.\overline{v}.w^{k-2}.
\end{eqnarray*}
Hodge index theorem says that the signature of $\mathcal{H}$ is $(1,h^{1,1}-1)$ where $h^{1,1}$ is the dimension of $H^{1,1}(X)$.

We are now ready for the proofs of Theorems \ref{TheoremFirstDynamicalDegreeIsSimple} and \ref{TheoremTheCaseNotStable}.

\begin{proof}[Proof of Theorem \ref{TheoremFirstDynamicalDegreeIsSimple}] 

First we show that there cannot be two non-collinear vectors $u_1,u_2\in H^{1,1}(X)$ for which $f^*u_1=\tau _1u_1$ and $f^*u_2=\tau _2u_2$, where $\tau =\min \{|\tau _1|, |\tau _2|\}>\sqrt{\lambda _2(f)}$. Assume otherwise, we will show that for any $u$ in the complex vector space of dimension $2$ generated by $u_1$ and $u_2$, then $\mathcal{H}(u,u)\geq 0$ and this gives a contradiction to the Hodge index theorem. To this end, it suffices to show that $u.\overline{u}$ is psef. Let $u=a_1u_1+a_2u_2$. For $n\in\mathbb{N}$, we define $$v_n=\frac{a_1}{\tau _1^n}u_1+\frac{a_2}{\tau _2^n}u_2.$$ Then it is easy to check that $(f^*)^n(v_n)=u$.  Because $f$ is $1$-stable, we have from Proposition \ref{LemmaPullPushFormulaForMeromorphicMaps} that
\begin{eqnarray*}
u.\overline{u}= (f^*)^n(v_n).(f^*)^n(\overline{v_n})=(f^n)^*(v_n).(f^n)^*(\overline{v_n})\geq (f^n)^*(v_n.\overline{v_n}),
\end{eqnarray*}
for any $n\in \mathbb{N}$. (Here the inequality $\geq$ means that the difference of the two cohomology classes is psef.) We fix an arbitrary norm $||\cdot ||$ on the vector space $H^{1,1}(X)$. Then $||v_n||$ is bounded by $1/\tau ^n$, hence the assumption that $\tau >\sqrt{\lambda _2(f)}$ implies that $(f^n)^*(v_n.\overline{v_n})$ converges to $0$. Therefore, $u.\overline{u}\geq 0$ as wanted. 

Hence $\lambda _1(f)$ is the unique eigenvalue of modulus $>\sqrt{\lambda _2(f)}$ of $f^*:H^{1,1}(X)\rightarrow H^{1,1}(X)$. It remains to show that $\lambda _1(f)$ is a simple root of the characteristic polynomial of $f^*:H^{1,1}(X)\rightarrow H^{1,1}(X)$. Assume otherwise, by using the Jordan normal form of a matrix, there will be two non-collinear vectors $u_1,u_2\in H^{1,1}(X)$ for which $f^*(u_1)=\lambda _1(f)u_1$ and $f^*(u_2)=\lambda _1(f)u_2+u_1$. Let $u=a_1u_1+a_2u_2$. For any $n\in \mathbb{N}$ we define $$v_n=\frac{a_1}{\lambda _1(f)^n}u_1-\frac{na_2}{\lambda _1(f)^{n+1}}u_1+\frac{a_2}{\lambda _1(f)^n}u_2.$$ Then it is easy to check that $(f^*)^n(v_n)=u$, and we can proceed as in the first part of the proof.
\end{proof}

\begin{proof}[Proof of Theorem \ref{TheoremTheCaseNotStable}]

1) First, we observe that for any $v\in H^{1,1}(X)$ with complex coefficients then $(f^*)^n(v).(f^*)^n(\overline{v})\geq (f^*)^n(v.\overline{v})$ for all $n\in \mathbb{N}$. For example, we show how to do this for $n=2$. Apply Proposition \ref{LemmaPullPushFormulaForMeromorphicMaps}, we have
\begin{eqnarray*} 
(f^*)^2(v).(f^*)^2(\overline{v})=f^*(f^*(v)).f^*(\overline{f^*(v)})\geq f^*(f^*(v).f^*(\overline{v})).
\end{eqnarray*}
By Proposition \ref{LemmaPullPushFormulaForMeromorphicMaps} again and the assumption that $f^*:H^{2,2}(X)\rightarrow H^{2,2}(X)$ preserves psef classes, we obtain 
\begin{eqnarray*}
f^*(f^*(v).f^*(\overline{v}))\geq (f^*)^2(v.\overline{v}),
\end{eqnarray*}
and hence $(f^*)^2(v).(f^*)^2(\overline{v})\geq (f^*)^2(v.\overline{v})$ as wanted. 

We now finish the proof of 1). Let $\omega _X$ be a K\"ahler form on $X$. Then from the first part of the proof we get 
\begin{eqnarray*}
(f^*)^n(\omega _X).(f^*)^n(\omega _X)\geq (f^*)^n(\omega _X^2),
\end{eqnarray*}
for all $n\in \mathbb{N}$. For convenience, we let $||\cdot ||$ denote an arbitrary norm on either $H^{1,1}(X)$ or $H^{2,2}(X)$. There is a constant $C>0$ independent of $n$, so that for all $n\in \mathbb{N}$, we have 
\begin{eqnarray*}
||(f^*)^n(\omega _X).(f^*)^n(\omega _X)||\leq C||(f^*)^n(\omega _X)||^2\leq C||(f^*)^n|_{H^{1,1}(X)}||^2,
\end{eqnarray*}
and 
\begin{eqnarray*}
C(f^*)^n(\omega _X^2)\geq ||(f^*)^n|_{H^{2,2}(X)}||.
\end{eqnarray*}
(In the second inequality we used the assumption that $f^*:H^{2,2}(X)\rightarrow H^{2,2}(X)$ preserves the cone of psef classes.)

Therefore, 
\begin{eqnarray*}
C^2||(f^*)^n|_{H^{1,1}(X)}||^2\geq ||(f^*)^n|_{H^{2,2}(X)}|| 
\end{eqnarray*}
for any $n\in \mathbb{N}$. Taking $n$-th root and letting $n\rightarrow\infty$, we obtain $r_1(f)^2\geq r_2(f)$.

2) Using the ideas from the proofs Theorem \ref{TheoremFirstDynamicalDegreeIsSimple} and 1), we obtain 2) immediately.

\end{proof}

\section{Analogs of Theorems \ref{TheoremFirstDynamicalDegreeIsSimple} and \ref{TheoremTheCaseNotStable} in the algebraic setting}

In this section we prove analogs of Theorems \ref{TheoremFirstDynamicalDegreeIsSimple} and \ref{TheoremTheCaseNotStable} in the algebraic setting. Throughout the section, we fix an algebraic closed field $K$ of characteristic $0$. Recall that a projective manifold over $K$ is a non-singular subvariety of a projective space $\mathbb{P}_K^N$. This section is organized as follows. In the first subsection we recall the definition and some results on algebraic cycles, the Chow's moving lemma and the Grothendieck-Hodge index theorem. In the second subsection we give definitions in the algebraic setting of dynamical degrees for rational maps and prove some basic properties of these dynamical degrees. In the last subsection we present the analogs of Theorems \ref{TheoremFirstDynamicalDegreeIsSimple} and \ref{TheoremTheCaseNotStable}. We stress that in the first two subsubsections, in particular in the definition of   dynamical degrees, we can work over any algebraic closed field, not necessarily of characteristic zero. 

\subsection{Algebraic cycles, Chow's moving lemma and Grothendieck-Hodge index theorem}
In the first subsubsection we recall some facts about algebraic cycles and the rational, algebraic and numerical equivalences. In the second and third subsubsections we recall Chow's moving lemma and Grothendieck-Hodge index theorem. In the last subsubsection we define some useful norms on the relevant vector spaces, which will be used to define dynamical degrees later. 

\subsubsection{Algebraic cycles}
Let $X\subset \mathbb{P}_K^N$ be a projective manifold of dimension $k$ over an algebraic closed field $K$ of characteristic zero. A $q$-cycle on $X$ is a finite sum $\sum n_i[V_i]$, where $V_i$ are $q$-dimensional irreducible subvarieties of $X$ and $n_i$ are integers. The group of $q$-cycles on $X$, denoted $Z_q(X)$, is the free abelian group on the $p$-dimensional subvarieties of $X$ (see Section 1.3 in Fulton \cite{fulton}). A $q$-cycle $\alpha$ is effective if it has the form
\begin{eqnarray*}
\alpha =\sum _i a_i[V_i],
\end{eqnarray*}
where $V_i$ are irreducible subvarieties of $X$ and $a_i\geq 0$.

Let $X$ and $Y$ be projective manifolds, and let $f:X\rightarrow Y$ be a morphism. For any irreducible subvariety $V$ of $X$, we define the pushforward $f_*[V]$ as follows. Let $W=f(V)$. If $dim(W)<dim (V)$, then $f_*[V]=0$. Otherwise, $f_*[V]=deg(V/W)[W]$. This gives a pushforward map $f_*:Z_q(X)\rightarrow Z_q(Y)$ (see Section 1.4 in \cite{fulton}). 

Let $p,f:X\times \mathbb{P}^1\rightarrow X, \mathbb{P}^1$ be the projections. Let $0=[0:1]$ and $\infty =[1:0]$ be the usual zero and infinity points of $\mathbb{P}^1$. We say that a cycle $\alpha$ in $Z_q(X)$ is rationally equivalent to zero if and only if there are $(q+1)$-dimensional irreducible subvarieties $V_1,...,V_t$ of $X\times \mathbb{P}^1$, such that the projections $f|_{V_i}:V_i\rightarrow \mathbb{P}^1$ are dominant, and
\begin{eqnarray*}     
\alpha =\sum _{i=1}^t([p_*(f|_{V_i}^{-1}(0))]-[p_*(f|_{V_i}^{-1}(\infty ))]).
\end{eqnarray*}
We call $V_{i,0}=[p_*(f|_{V_i}^{-1}(0))]$ and $V_{i,\infty}=[p_*(f|_{V_i}^{-1}(\infty ))]$  the specializations of $V_i$ at $0$ and $\infty$. Let $Rat_q(X)$ be the group of $q$-cycles rationally equivalent to zero. The group of $q$-cycles modulo rational equivalence on $X$ is the factor group 
\begin{eqnarray*}
A_q(X)=Z_q(X)/Rat _q(X).
\end{eqnarray*}
(See Section 1.6 in \cite{fulton}.) 

We say that a cycle $\alpha$ in $A_q(X)$ is algebraically equivalent to zero if and only if there is a non-singular variety $T$ of dimension $m$, points $t_1,t_2\in T$ which are rational over the ground field $K$, a cycle $\beta$ in $A_{k+m}(X)$ such that
\begin{eqnarray*}
  \alpha =\beta _{t_1}-\beta _{t_2},
\end{eqnarray*}
where $\beta _{t_i}$'s are specializations of $\beta$ at $t_i$'s. The group of $q$-cycles modulo algebraic equivalence on $X$ is denoted by $B_q(X)$ (see Sections 10.1 and 10.3 in \cite{fulton}).  

We write $Z^p(X)$, $A^p(X)$ and $B^p(X)$ for the corresponding groups of cycles of codimension $p$. Since $X$ is smooth, we have an intersection product $A^p(X)\times A^q(X)\rightarrow A^{p+q}(X)$, making $A^*(X)$ a ring, called the Chow's ring of $X$ (see Sections 8.1 and 8.3 in \cite{fulton}).  

For a dimension $0$ cycle $\gamma =\sum _im_i[p_i]$ on $X$, we define its degree to be $deg(\gamma )=\sum _im_i$. We say that a cycle $\alpha \in A^{p}(X)$ is numerically equivalent to zero if and only $deg(\alpha .\beta )=0$ for all $\beta \in A^{k-p}(X)$ (see Section 19.1 in \cite{fulton}). The group of codimension $p$ algebraic cycles modulo numerical equivalence is denoted by $N^p(X)$. These are finitely generated free abelian groups (see Example 19.1.4 in \cite{fulton}). The first group $N^1(X)$ is a quotient of the Neron-Severi group $NS(X)=B^1(X)$. The latter is also finitely generated, as proved by Severi and Neron. We will use the vector spaces $N^p_{\mathbb{R}}(X)=N^p(X)\otimes _{\mathbb{Z}}\mathbb{R}$ and $N^p_{\mathbb{C}}(X)=N^p(X)\otimes _{\mathbb{Z}}\mathbb{C}$ in defining dynamical degrees and in proving analogs of Theorems \ref{TheoremFirstDynamicalDegreeIsSimple} and \ref{TheoremTheCaseNotStable}.  

{\bf Remarks.} We have the following inclusions: rational equivalence $\subset $ algebraic equivalence $\subset$ numerical equivalence.

\subsubsection{Chow's moving lemma}

Let $X$ be a projective manifold of dimension $k$ over $K$. If $V$ and $W$ are two irreducible subvarieties of $X$, then either $V\cap W=\emptyset$ or any irreducible component of $V\cap W$ has dimension at least $dim (V)+dim(W)-k$. We say that $V$ and $W$ are properly intersected if any component of $V\cap W$ has dimension exactly $dim (V)+dim (W)-k$. When $V$ and $W$ intersect properly, the intersection $V.W$ is well-defined as an effective $dim(V)+dim(W)-k$ cycle. 

Given $\alpha =\sum _{i}m_i[V_i]\in Z_q(X)$ and $\beta =\sum _{j}n_j[W_j]\in Z_{q'}(X)$, we say that $\alpha .\beta$ is well-defined if every component of $V_i\cap W_j$ has the correct dimension. Chow's moving lemma says that we can always find $\alpha '$ which is rationally equivalent to $\alpha$ so that $\alpha '.\beta $ is well-defined. Since in the sequel we will need to use some specific properties of such cycles $\alpha '$, we recall here a construction of such cycles $\alpha '$, following the paper Roberts \cite{roberts}. See also the paper Friedlander-Lawson \cite{friedlander-lawson} for a generalization to moving families of cycles of bounded degrees. 

Fixed an embedding $X\subset \mathbb{P}^N_K$, we choose a linear subspace $L\subset \mathbb{P}^N_K$ of dimension $N-k-1$ such that $L\cap X=\emptyset$. For any irreducible subvariety $Z$ of $X$ we denote by $C_L(Z)$ the cone over $Z$ with vertex $L$ (see Example 6.17 in the  book Harris \cite{harris}). For any such $Z$, $C_L(Z).X$ is well-defined and has the same dimension as $Z$, and moreover $C_L(Z).X-Z$ is effective (see Lemma 2 in \cite{roberts}).

Let $Y_1,Y_2,\ldots , Y_m$ and $Z$ be irreducible subvarieties of $X$. We define the excess $e(Z)$ of $Z$ relative to $Y_1,\ldots ,Y_m$ to be the maximum of the integers $$dim (W)+k-dim (Z)-dim (Y_i),$$ 
where $i$ runs from $1$ to $m$, and $W$ runs through all components of $Z\cap Y_i$, provided that one of these integers is non-negative. Otherwise, the excess is defined to be $0$. 

More generally, if $Z=\sum _im_i[Z_i]$ is a cycle, where $Z_i$ are irreducible subvarieties of $X$, we define $e(Z)=\max _ie(Z_i)$. We then also define the cone $C_L(Z)=\sum _im_iC_{L}(Z_i)$. 
         
The main lemma (page $93$) in \cite{roberts} says that for any cycle $Z$ and any irreducible subvarieties $Y_1,\ldots ,Y_m$, then $(e(C_L(Z).X-Z))\leq \max (e(Z)-1,0)$ for generic linear subspace $L\subset \mathbb{P}^N$ of dimension $N-k-1$ such that $L\cap X=\emptyset$.

Now we can finish the proof of Chow's moving lemma as follows (see Theorem page 94 in \cite{roberts}). Given $Y_1,\ldots ,Y_m$ and $Z$ be irreducible varieties on $X$. If $e=e(Z)=0$ then $Z$ intersect properly $Y_1,\ldots ,Y_m$, hence we are done. Otherwise, $e\geq 1$. Applying the main lemma, we can find linear subspaces $L_1,\ldots ,L_e\subset \mathbb{P}^N_K$ of dimension $N-k-1$, such that if $Z_0=Z$ and $Z_i=C_{L_i}(Z_{i-1}).X-Z_{i-1}$ for $i=1,\ldots ,e=e(Z)$, then $e(Z_i)\leq e-i$. In particular, $e(Z_e)=0$. It is easy to see that
\begin{eqnarray*}
Z=Z_0=(-1)^eZ_e+\sum _{i=1}^{e}(-1)^{i-1}C_{L_i}(Z_{i-1}).X.    
\end{eqnarray*} 

It is known that there are points $g\in Aut(\mathbb{P}_K^N)$ such that $(gC_{L_i}(Z_{i-1})).X$ and $(gC_{L_i}(Z_{i-1})).Y_j$ are well-defined for $i=1,\ldots ,e$ and $j=1,\ldots ,m$. We can choose a rational curve in $Aut(\mathbb{P}_K^N)$ joining the identity map $1$ and $g$, thus see that $Z$ is rationally equivalent to
\begin{eqnarray*}
Z'=(-1)^eZ_e+\sum _{i=1}^{e}(-1)^{i-1}(gC_{L_i}(Z_{i-1})).X. 
\end{eqnarray*}
By construction, $e(Z')=0$, as desired. 

\subsubsection{Grothendieck-Hodge index theorem}

Let $X\subset \mathbb{P}^N_K$ be a projective manifold of dimension $k$. Let $H\subset \mathbb{P}^N_K$ be a hyperplane, and let $\omega _X =H|_X$. We recall that $N^p(X)$, the group of codimension $p$ cycles modulo the numerical equivalence, is a finitely generated free abelian group. We define $N^p_{\mathbb{R}}(X)=N^p(X)\otimes _{\mathbb{Z}}\mathbb{R}$ and $N^p_{\mathbb{C}}(X)=N^p(X)\otimes _{\mathbb{Z}}\mathbb{C}$. These are real (and complex) vector spaces of real (and complex) dimension equal $rank(N^p(X))$. For $p=1$, it is known that $dim _{\mathbb{R}} (N^1_{\mathbb{R}}(X)) =rank(NS(X))=:\rho$, the rank of the Neron-Severi group of $X$ (see Example 19.3.1 in \cite{fulton}). 

We define for $u,v\in N^1_{\mathbb{C}}(X)$ the Hermitian form 
\begin{eqnarray*}
\mathcal{H}(u,v)=deg(u.\overline{v}.\omega _X^{k-2}).
\end{eqnarray*}
Here the degree of a complex $0$-cycle $\alpha +i\beta$ is defined to be the complex number $deg(\alpha )+ideg(\beta )$. The analog of Hodge index theorem is the Grothendieck-Hodge index theorem, which says that $\mathcal{H}$ has signature $(1,\rho -1)$. For the convenience of the reader, we recall a sketch of the proof of the theorem here. We thank Claire Voisin for helping with this. First, observe that we can reduce the result to the case where $X$ is a surface, i.e. $dim (X)=2$. In fact, by Bertini's theorem, for generic $k-2$ ample hypersurfaces in $|H|$, their intersection is a smooth surface $\Sigma$, and
\begin{eqnarray*}  
\mathcal{H}(u,v)=deg(u|_{\Sigma}.\overline{v}|_{\Sigma})=\mathcal{H}_{\Sigma}(u|_{\Sigma},v|_{\Sigma}).   
\end{eqnarray*}
The latter is the corresponding Hermitian form on the surface $\Sigma$. The Grothendieck-Lefschetz theorem gives that the restriction of Neron-Severi groups $NS(X)\rightarrow NS(\Sigma )$ is injective, and so is the restriction map $N^1_{\mathbb{C}}(X)\rightarrow N^1_{\mathbb{C}}(\Sigma )$ (see Example 19.3.3 in \cite{fulton}). Hence we showed that the Grothendieck-Hodge index theorem is proved if it can be proved for surfaces. The latter case is well-known, see e.g. the paper Grothendieck \cite{grothendieck}.

\subsubsection{Some norms on the vector spaces $N^p_{\mathbb{R}}(X)$ and $N^p_{\mathbb{C}}(X)$}

Given $\iota :X\subset \mathbb{P}^N_K$ a projective manifold of dimension $k$, let $H\in A^1(\mathbb{P}^N)$ be a hyperplane and $\omega _X=H|_{X}=\iota ^*(H)\in A^1(X)$. For an irreducible subvariety $V\subset X$ of codimension $p$, we define the degree of $V$ to be $deg(V)=$ the degree of the dimension $0$ cycle $V.\omega _X^{k-p}$, or equivalently $deg(V)=$degree of the variety $\iota _*(V)\subset \mathbb{P}^N$. Similarly, we define for an effective codimension $p$ cycle $V=\sum _{i}m_i[V_i]$ (here $m_i\geq 0$ and $V_i$ are irreducible), the degree $deg(V)=\sum _im_ideg(V_i)$. This degree is extended to vectors in $N^p_{\mathbb{R}}(X)$. Note that the degree map is a numerical equivalent invariant.   

As a consequence of Chow's moving lemma, we have the following result on intersection of cycles

\begin{lemma}
Let $V$ and $W$ be irreducible subvarieties in $X$. Then the intersection $V.W\in A^*(X)$ can  be represented as $V.W=\alpha _1-\alpha _2$, where $\alpha _1,\alpha _2\in A^*(X)$ are effective cycles and $deg(\alpha _1),\deg(\alpha _2)\leq Cdeg(V)deg(W)$, where $C>0$ is a constant independent of $V$ and $W$. 
\label{LemmaDegreeOfIntersections}\end{lemma}
\begin{proof}
Using Chow's moving lemma, $W$ is rationally equivalent to 
\begin{eqnarray*}
W'=\sum _{i=1}^e(-1)^{i-1}gC_{L_i}(W_{i-1}).X+(-1)^eW_e,
\end{eqnarray*}
where $W_0=W$, $W_i=C_{L_i}(W_{i-1}).X-W_{i-1}$, $C_{L_i}(W_{i-1})\subset \mathbb{P}^N_K$ is a cone over $W_{i-1}$, and $g\in Aut(\mathbb{P}^N_K)$ is an automorphism. Moreover, $gC_{L_i}(W_{i-1}).X$, $gC_{L_i}(W_{i-1}).V$ and $W_e.V$ are all well-defined. We note that $e\leq k=dim (X)$, and for any $i=1,\ldots ,e$ 
\begin{eqnarray*} 
deg(W_i)&\leq& deg (gC_{L_i}(W_{i-1}).X)\leq deg(gC_{L_i}(W_{i-1}))deg(X)\\
&=&deg(C_{L_i}(W_{i-1})).deg(X)=deg(W_{i-1})deg(X).
\end{eqnarray*}
Here we used that $deg(C_{L_i}(W_{i-1})=deg(W_{i-1})$ (see Example 18.17 in \cite{harris}), and $deg(gC_{L_i}(W_{i-1})=deg(C_{L_i}(W_{i-1})$ because $g$ is an automorphism of $\mathbb{P}^N$ (hence a linear map).  

Therefore, the degrees of $W_i$ are all $\leq (deg(X))^kdeg(W)$. By definition, the intersection product $V.W\in A^*(X)$ is given by $V.W'$, which is well-defined. We now estimate the degrees of each effective cycle $gC_{L_i}(W_{i-1})|_X.V$ and $W_e.V$. Firstly, we have by the projection formula
\begin{eqnarray*}
deg(gC_{L_i}(W_{i-1})|_X.V)&=&deg(\iota _*(gC_{L_i}(W_{i-1})|_X.V))=deg (gC_{L_i}(W_{i-1}).\iota _*(V))\\
&=&deg(C_{L_i}(W_{i-1})).deg (V)\leq deg(X)^kdeg(W)deg(V).
\end{eqnarray*}  
Finally, we estimate the degree of $W_e.V$. Since $W_e.V$ is well-defined, we can choose a linear subspace $L\subset \mathbb{P}^N$ so that $C_L(W_e).X$ and $C_L(W_e).V$ are well-defined. Recall that $C_L(W_e)-W_e$ is effective, we have
\begin{eqnarray*}
deg(V.W_e)\leq deg(V.C_L(W_e)|_{X})=deg (V).deg(C_L(W_e))\leq deg(X)^kdeg(V)deg(W). 
\end{eqnarray*}

From these estimates, we see that we can write
\begin{eqnarray*}
V.W'=\alpha _1-\alpha _2,
\end{eqnarray*}
where $\alpha _1,\alpha _2$ are effective cycles and $deg(\alpha _1),deg(\alpha _2)\leq Cdeg(V)deg(W)$, where $C=k.deg(X)^k$ is independent of $V$ and $W$.
\end{proof}

Using this degree map, we define for an arbitrary vector $v\in N^p_{\mathbb{R}}(X)$, the norm 
\begin{equation}
||v||_1=\inf \{deg(v_1)+deg(v_2):~v=v_1-v_2,~v_1, v_2\in N^p_{\mathbb{R}}(X) \mbox{ are effective}\}.
\label{LabelDefinitionNorm}\end{equation}
We check that this is actually a norm. It is easy to check that $||\lambda v||_1=|\lambda | ||v||_1$ for any $\lambda \in \mathbb{R}$ and $v\in N^p_{\mathbb{R}}(X)$. The triangle inequality is also easy to prove. It remains to check that if $||v||_1=0$ then $v=0$. In fact, if $||v||_1=0$ then by definition there are sequences $v_{1,n},v_{2,n}\in N^p_{\mathbb{R}}(X)$ of effective cycles so that $v=v_{1,n}-v_{2,n}$ and $deg(v_{1,n}),deg(v_{2,n})\rightarrow 0$. From Lemma \ref{LemmaDegreeOfIntersections}, we have that for any $w\in N^{k-p}_{\mathbb{R}}(X)$ 
\begin{eqnarray*} 
deg(v.w)=\lim _{n\rightarrow\infty}deg((v_{1,n}-v_{2,n}).w)=0.
\end{eqnarray*}
Hence $v=0$, since from the definition of $N^p(X)$, the bilinear form $N^p(X)\times N^{k-p}(X)\rightarrow \mathbb{Z}$, $(v,w)\mapsto deg(v.w)$ is non-degenerate. (In fact, let us choose a basis $(v_i)_{i\in I}$ for $N^p(X)$ and a basis $(w_j)_{j\in J}$ for $N^{k-p}(X)$. These are also bases for the corresponding real vector spaces. Let $v=\sum _{i}a_iv_i$, where $a_i\in \mathbb{R}$. Now $deg(v.w)=0$ for every $w\in N^{k-p}_{\mathbb{R}}(X)$ if and only if $deg(v.w_j)=0$ for every $j\in J$. The latter is a system of homogeneous equations in $a_i$ with integer coefficients $deg(v_i.w_j)$, therefore it has a non-trivial solution $(a_i)\in \mathbb{R}^I$ if and only if it has a non-trivial solution $(a_i)\in \mathbb{Z}^I$. But there is no non-trivial solution $(a_i)\in \mathbb{Z}^I$ to the system because the bilinear form $N^p(X)\times N^{k-p}(X)\rightarrow \mathbb{Z}$, $(v,w)\mapsto deg(v.w)$ is non-degenerate. Hence there is no non-trivial solution $(a_i)\in \mathbb{R}^I$ to the system, i.e. if $v\in N^{p}_{\mathbb{R}}(X)$ such that $deg(v.w)=0$ for all $w\in N^{k-p}_{\mathbb{R}}(X)$ then $v=0$.)

{\bf Remark.} It is easy to check that if $v\in N^p_{\mathbb{R}}(X)$ is effective, then $||v||_1=deg(v)$. Since $N^p_{\mathbb{R}}(X)$ is of finite dimensional, any norm on it is equivalent to $||\cdot ||_1$. We can also complexify these norms to define norms on $N^p_{\mathbb{C}}(X)$.

\subsection{Dynamical degrees and $p$-stability}
In the first subsubsection we consider pullback and strict transforms of algebraic cycles by rational maps. In the second subsubsection we define dynamical degrees and prove some of their basic properties. In the last subsubsection we define $p$-stability. 

\subsubsection{Pullback and strict transforms of algebraic cycles by rational maps}
Let $X$ and $Y$ be two projective manifolds and $f:X\rightarrow Y$ a dominant rational map. Then we can define the pushforward operators $f_*:A_q(X)\rightarrow A_q(Y)$ and pullback operators $f^*:A^p(Y)\rightarrow A^p(X)$ (see Chapter 16 in \cite{fulton}). For example, there are two methods to define the pullback operators: 

Method 1: Let $\pi _X,\pi _Y:X\times Y\rightarrow X,Y$ be the two projections, and let $\Gamma _f$ be the graph of $f$. For $\alpha \in A^p(Y)$, we define $f^*(\alpha )\in A^p(X)$ by the following formula
\begin{eqnarray*}
f^*(\alpha )=(\pi _X)_*(\Gamma _f.\pi _Y^*(\alpha )). 
\end{eqnarray*}

Method 2: Let $\Gamma \rightarrow \Gamma _f$ be a resolution of singularities of $\Gamma _f$, and let $p,g:\Gamma \rightarrow X,Y$ be the induced morphisms. then we define 
\begin{eqnarray*}
f^*(\alpha )=p_*(g^*(\alpha )). 
\end{eqnarray*}

For the convenience of the readers, we recall here the arguments to show the equivalences of these two methods. Firstly, we show that the definition in Method 2 is independent of the choice of the resolution of singularities of $\Gamma _f$. In fact, let $\Gamma _1,\Gamma _2\rightarrow \Gamma _f$ be two resolutions of $\Gamma _f$ with the induced morphisms $p_1,g_1$ and $p_2,g_2$. Then there is another resolution of singularities $\Gamma \rightarrow \Gamma _f$ which dominates both $\Gamma _1$ and $\Gamma _2$ (e.g. $\Gamma$ is a resolution of singularities of the graph of the induced birational map $\Gamma _1\rightarrow \Gamma _2$). Let $\tau _1,\tau _2:\Gamma \rightarrow \Gamma _1,\Gamma _2$ the corresponding morphisms, and $p=p_1\circ \tau _1=p_2\circ \tau _2:\Gamma \rightarrow X$ and $g=g_1\circ \tau _1=g_2\circ \tau _2:\Gamma \rightarrow Y$ the induced morphisms. For $\alpha \in A^p(Y)$, we will show that $(p_1)_*(g_1^*\alpha )=p_*(g^*\alpha )=(p_2)_*(g_2^*\alpha )$. We show for example the equality $(p_1)_*(g_1^*\alpha )=p_*(g^*\alpha )$. In fact, we have by the projection formula 
\begin{eqnarray*}
p_*g^*(\alpha )&=&(p_1\circ \tau _1)_*(g_1\circ \tau _1)^*\alpha\\
&=&(p_1)_*(\tau _1)_*(\tau _1)^*(g_1)^*(\alpha )\\
&=&(p_1)_*(g_1)^*(\alpha ),
\end{eqnarray*}
as wanted. Finally, we show that the definitions in Method 1 and Method 2 are the same. By the embedded resolution of singularities (see e.g. the book \cite{kollar}), there is a finite blowup $\pi :\widetilde{X\times Y}\rightarrow X\times Y$ so that the strict transform $\Gamma$ of $\Gamma _f$ is smooth. Hence $\Gamma$ is a resolution of singularities of $\Gamma _f$, and $p=\pi _X\circ \pi \circ \iota , g=\pi _Y\circ \pi \circ\iota :\Gamma \rightarrow X,Y$ are the induced maps, where $\iota :\Gamma \subset \widetilde{X\times Y}$ is the inclusion map. For $\alpha \in A^p(Y)$, we have by the projection formula
\begin{eqnarray*}
p_*g^*(\alpha )&=&(\pi _X)_*\pi _* \iota _*\iota ^*\pi ^*\pi _Y^*(\alpha )=(\pi _X)_*\pi _* [\pi ^*\pi _Y^*(\alpha ).\Gamma ]\\
&=&(\pi _X)_*[\pi _Y^*(\alpha ).\pi _*(\Gamma )]=(\pi _X)_*[\pi _Y^*(\alpha ).\Gamma _f],
\end{eqnarray*}   
as claimed. 

In defining dynamical degrees and proving some of their basic properties, we need to estimate the degrees of the pullback and of strict transforms by a meromorphic map of a cycle. We present these estimates in the remaining of this subsubsection. We fix a resolution of singularities $\Gamma $ of the graph $\Gamma _f$, and let $p,g:\Gamma \rightarrow X,Y$ be the induced morphisms. By the theorem on the dimension of fibers (see e.g. the corollary of Theorem 7 in Section 6.3 Chapter 1 in the book Shafarevich \cite{shafarevich}), the sets
\begin{eqnarray*} 
V_{l}=\{y\in Y:~dim(g^{-1}(y))\geq l\}
\end{eqnarray*}
are algebraic varieties of $Y$. We denote by $\mathcal{C}_g=\cup _{l>dim (X)-dim (Y)}V_l$ the critical image of $g$. We have the first result considering the pullback of a subvariety of $Y$

\begin{lemma}
Let $W$ be an irreducible subvariety of $Y$. If $W$ intersects properly any irreducible component of $V_l$ (for any $l>dim(X)-dim(Y)$), then $g^*[W]=[g^{-1}(W)]$ is well-defined as a subvariety of $\Gamma$. Moreover this variety represents the pullback $g^*(W)$ in $A^*(\Gamma )$.  
\label{LemmaGoodPullbackVariety}\end{lemma}
\begin{proof} (See also Example 11.4.8 in \cite{fulton}.) By the intersection theory (see Section 8.2 in \cite{fulton} and Theorem 3.4 in \cite{friedlander-lawson}), it suffices to show that $g^{-1}(W)$ has the correct dimension $dim(X)-dim(Y)+dim(W)$. First, if $y\in W-\mathcal{C}_g$ then $dim(g^{-1}(y))=dim(X)-dim(Y)$ by definition of $\mathcal{C}_g$. Hence $dim (g^{-1}(W-\mathcal{C}_g)=dim (W)+dim(X)-dim (Y)$. It remains to show that $g^{-1}(W\cap \mathcal{C}_g)$ has dimension $\leq dim(X)+dim(W)-dim (Y)-1$. Let $Z$ be an irreducible component of $W\cap \mathcal{C}_g$. We define $l=\inf \{dim(g^{-1}(y)):~y\in Z\}$. Then $l>dim (X)-dim(Y)$ and for generic $y\in Z$ we have $dim(g^{-1}(y))=l$ (see Theorem 7 in Section 6.3 in Chapter 1 in \cite{shafarevich}). Let $V\subset V_l$ be an irreducible component containing $Z$.  By assumption $V.W$ has dimension $dim(V)+dim(W)-dim(Y)$, hence $dim(Z)\leq dim (V)+dim(W)-dim (Y)$. We obtain $$dim(g^{-1}(Z-V_{l+1}))=l+dim(Z)\leq l+dim(V)+dim(W)-dim (Y).$$ Since $g$ is surjective (because $f$ is dominant) and $V\not= Y$, it follows that $$dim (X)-1\geq dim (g^{-1}(V))\geq dim (V)+l.$$ From these last two estimates we obtain 
\begin{eqnarray*}
dim (g^{-1}(Z-V_{l+1}))&=&l+dim(V)+dim(W)-dim (Y)\\
&\leq& dim (X)-1+dim (W)-dim (Y).
\end{eqnarray*}
Since there are only a finite number of such components, it follows that $dim(g^{-1}(W\cap \mathcal{C}_g))\leq dim (W)+dim (X)-dim(Y)-1$, as claimed. 
\end{proof} 
 
We next estimate the degree of the pullback of a cycle. Fix an embedding $Y\subset \mathbb{P}^N_K$, and let $\iota :Y\subset \mathbb{P}^N_K$ the inclusion. Let $H\subset \mathbb{P}^N_K$ be a generic hyperplane and let $\omega _Y=H|_{Y}$. 
 
\begin{lemma} 
a) Let $p=0,\ldots ,dim (Y)$, and let $Z\subset X$ be a proper subvariety. Then there is a linear subspace $H^p\subset \mathbb{P}^N_K$ of codimension $p$ such that $H^p$ intersects $Y$ properly, $f^*(\iota ^*(H^p))$ is well-defined as a subvariety of $X$, and $f^*(\iota ^*(H^p))$ has no component on $Z$. In particular, for any non-negative integer $p$, the pullback $f^*(\omega _Y^p)\in A^{p}(X)$ is effective. 

b) Let $W$ be an irreducible of codimension $p$ in $Y$. Then in $A^p(X)$, we can represent $f^*(W)$ by $\beta _1-\beta _2$, where $\beta _1$ and $\beta _2$ are effective and $\beta _1,\beta _2\leq Cdeg(W)f^*(\omega _Y^p)$ for some constant $C>0$ independent of the variety $W$, the manifold $X$ and the map $f$.
\label{LemmaDegreeOfPullback}\end{lemma} 
\begin{proof}

Since by definition $f^*(W)=p_*g^*(W)$ and since $p_*$ preserves effective classes, it suffices to prove the lemma for the morphism $g$. We let the varieties $V_l$ as those defined before Lemma \ref{LemmaGoodPullbackVariety}.  

a) Let $H^p\subset \mathbb{P}^N_K$ be a generic codimension $p$ linear subspace. Then in $A^p(Y)$, $\omega _Y^p$ is represented by $\iota ^*(H^p)$. We can choose such an $H^p$ so that $H^p$ intersects properly $Y$, $g(Z)$ and all irreducible components of $V_l$ and $g(Z)\cap V_l$ for all $l>dim(X)-dim(Y)$. By Lemma \ref{LemmaGoodPullbackVariety}, the pullback $g^*(\iota ^*(H^p))=g^{-1}(\iota ^*(H^p))$ is well-defined as a subvariety of $\Gamma$. Moreover, the dimension of $g^{-1}(\iota ^*(H^p))\cap Z$ is less than the dimension of $g^{-1}(\iota ^*(H^p))$. In particular, $g^*(\iota ^*(H^p))$ is effective and has no component on $Z$. 

b) By Chow's moving lemma, $W$ is rationally equivalent to $\iota^*(\alpha _1)-\iota^*(\alpha _2)\pm W_e$, where $\alpha _1,\alpha _2\subset \mathbb{P}^N_K$  and $W_e\subset Y$ are subvarieties of codimension $p$, and they intersect properly $Y$ and all irreducible components of $V_l$ for all $l>dim(X)-dim(Y)$. Moreover, $deg(\alpha _1),deg(\alpha _2),deg(W_e)\leq Cdeg(W)$, for some $C>0$ independent of $W$. By the proof of Chow's moving lemma, we can find a codimension $p$ variety $\alpha\subset \mathbb{P}^N_K$ so that $\alpha$ intersect properly with $Y$ and all $V_l$, $\iota ^*(\alpha ) -W_e$ is effective, and $deg(\alpha )\leq Cdeg(W_e)$. Note that in $A^p(Y)$ we have $\iota ^*(\alpha _1)\sim deg(\alpha _1)\omega _Y^p$, $\iota ^*(\alpha _2)\sim deg(\alpha _2)\omega _Y^p$ and $\iota ^*(\alpha )\sim deg(\alpha )\omega _Y^p$. Note also that $0\leq g^*(W_e)\leq g^*(\iota ^*(\alpha ))$. Therefore, in $A^p(\Gamma )$
\begin{eqnarray*}
g^*(W)\sim deg(\alpha _1)g^*(\omega _Y^p)-deg (\alpha _2)g^*(\omega _Y^p)\pm g^*(W_e),    
\end{eqnarray*}
where each of the three terms on the RHS is effective and $\leq Cdeg(W)g^*(\omega _Y ^p)$ for some $C>0$ independent of $W$, $X$ and $f$.
\end{proof}
\begin{lemma}
Let $f:X\rightarrow Y$ be a rational map. For any $p=0,\ldots ,dim(Y)-1$, we have
$$
f^{*}(\omega _Y^{p+1})\leq f^*(\omega _Y^p).f^*(\omega _Y)
$$
in $A^{p+1}(X)$.
\label{LemmaBoundForPullbackKahlerForms}\end{lemma}  
\begin{proof}
Let $Z\subset X$ be a proper subvariety containing $p(g^{-1}(\mathcal{C}_g))$ so that $p:\Gamma -p^{-1}(Z)\rightarrow X-Z$ is an isomorphism. Then the restriction map $$p_0:\Gamma -g^{-1}(gp^{-1}(Z))\rightarrow X-p(g^{-1}(gp^{-1}(Z)))$$ is also an isomorphism, and the restriction map $$g_0:\Gamma -g^{-1}(gp^{-1}(Z))\rightarrow Y- gp^{-1}(Z)$$ has fibers of the correct dimension $dim(X)-dim(Y)$. 

Choose $H,H^p,H^{p+1}\subset \mathbb{P}^N_K$ be linear subspaces of codimension $1$, $p$ and $p+1$ such that $H^{p+1}=H\cap H^p$. We can find an automorphism $\tau \in \mathbb{P}^N_K$, so that $\tau (H),\tau (H^p), \tau (H^{p+1})$ intersects properly $Y$ and all irreducible components of $gp^{-1}(Z)$ and of  $V_l$ for all $l>dim (X)-dim(Y)$. For convenience, we write $H,H^p$ and $H^{p+1}$ for $\tau (H),\tau (H^p), \tau (H^{p+1})$, and $H_Y,H_Y^p$ and $H_Y^{p+1}$ for their intersection with $Y$. Then all the varieties $g^{-1}(H_Y),g^{-1}(H_Y^p)$ and $g^{-1}(H_Y^{p+1})$ have the correct dimensions, and have no components in $g^{-1}(gp^{-1}(Z))$. Hence the pullbacks $f^*(H_Y), f^*(H_Y^p)$ and $f^*(H_Y^{p+1})$ are well-defined as varieties in $X$ and has no components on $p(g^{-1}(gp^{-1}(Z)))$.

We next observe that the two varieties $f^*(H_Y)$ and $f^*(H_Y^p)$ intersect properly. Since $f^*(H_Y)$ is a hypersurface, it suffices to show that any component of $f^*(H_Y)\cap f^*(H_Y^p)$ has codimension $p+1$. Since $f^*(H_Y^p)$ has no component on $p(g^{-1}(gp^{-1}(Z)))$, the codimension of $f^*(H_Y)\cap f^*(H_Y^p)\cap p(g^{-1}(gp^{-1}(Z)))$ is at least $p+1$. It remains to show that $f^*(H_Y)\cap f^*(H_Y^p)\cap (X-p(g^{-1}(gp^{-1}(Z))))$ has codimension $p+1$. Since $p_0$ is an isomorphism, the codimension of the latter equals that of 
$$g^{-1}(H_Y)\cap g^{-1}(H_Y^p)\cap (\Gamma - g^{-1}(gp^{-1}(Z)))=g^{-1}(H_Y\cap H_Y^{p})\cap (\Gamma - g^{-1}(gp^{-1}(Z)))$$ 
which is $p+1$.

Therefore $f^*(H_Y).f^*(H_Y^{p+1})$ is well-defined as a variety of $X$, and on $X-p(g^{-1}(gp^{-1}(Z)))$ it equals
\begin{eqnarray*}
(p_0)*(g_0^*(H_Y).g_0^*(H_Y^p))=(p_0)_*(g_0^*(H_Y^{p+1}))=p_*g^*(H_Y^{p+1}).
\end{eqnarray*}
Since the latter has no component on $p(g^{-1}(gp^{-1}(Z)))$, it follows that $f^*(H_Y).f^*(H_Y^p)\geq f^*(H_Y^{p+1})$. From this inequality, we obtain the desired inequality in $A^{p+1}(X)$
\begin{eqnarray*}
f^*(\omega _Y^p).f^*(\omega _Y)\geq f^*(\omega _Y^{p+1}).
\end{eqnarray*} 
\end{proof}

Finally, we estimate the degree of a strict transform of a cycle. Define $$g_0=g|_{\Gamma -g^{-1}(\mathcal{C}_g)}: \Gamma -g^{-1}(\mathcal{C}_g)\rightarrow Y-\mathcal{C}_g.$$ Then $g_0$ is a proper morphism, and for any $y\in Y-\mathcal{C}_g$, $g_0^{-1}(y)$ has the correct dimension $dim(X)-dim(Y)$. Let $W\subset Y$ be a codimension $p$ subvariety. The inverse image $g_0^{-1}(W)=g^{-1}(W)\cap (\Gamma -g^{-1}(\mathcal{C}_g))\subset \Gamma -g^{-1}(\mathcal{C}_g)$ is a closed subvariety of codimension $p$ of $\Gamma -g^{-1}(\mathcal{C}_g)$, hence its closure $cl(g_0^{-1}(W))\subset \Gamma$ is a subvariety of codimension $p$, and we define $f^{o}(W)=p_*cl(g_0^{-1}(W))$. Note that a strict transform depends on the choice of a resolution of singularities $\Gamma$ of the graph $\Gamma _f$. (We can also define a strict transform more intrinsically using the graph $\Gamma _f$ directly, as in \cite{dinh-sibony3}.)

\begin{lemma}
Let $W\subset Y$ be a codimension $p$ subvariety. Then $f^o(W)$ is an effective cycle, and in $A^p(X)$
\begin{eqnarray*} 
f^{o}(W)\leq Cdeg(W)f^*(\omega _Y^p),
\end{eqnarray*}
where $C>0$ is a constant independent of the the variety $W$, the manifold $X$  and the map $f$. 
\label{LemmaDegreeOfStrictTransform}\end{lemma}
\begin{proof}
That $f^0(W)$ is an effective cycle follows from the definition. It suffices to prove the lemma for the morphism $g:\Gamma \rightarrow Y$. By the proof of Chow's moving lemma, we can decompose $W$ as follows
\begin{eqnarray*}
W=\sum _{i=1}^e(-1)^{i-1}\iota ^*(C_{L_i}(W_{i-1}))+(-1)^eW_e,
\end{eqnarray*} 
where the variety $W_e$ intersects properly all irreducible components of $V_l$ for all $l>dim(X)-dim(Y)$, and $C_{i}(W_{i-1})\subset \mathbb{P}^N_K$ are subvarieties of codimension $p$ intersecting $Y$ properly (but may not intersect properly the irreducible components of $V_l$). Moreover, we have the following bound on the degrees
\begin{equation}
deg(W_e),deg(C_{L_i}(W_{i-1}))\leq Cdeg(W),
\label{Equation0}\end{equation}
for all $i$, where $C>0$ is independent of $W$, $X$ and $f$. 

By the definition of $g^0$ we have
\begin{equation}
g^o(W)= \sum _{i=1}^e(-1)^{i-1}g^o(\iota ^*(C_{L_i}(W_{i-1})))+(-1)^eg^o(W_e).
\label{Equation1}\end{equation}

Note that $e\leq dim(Y)$. We now estimate each term on the RHS of (\ref{Equation1}). Let $S\subset \mathbb{P}^N_K$ be a subvariety of codimension $p$ intersecting $Y$ properly (but may not intersecting properly the components of $V_l$). We first show that for any such $S$
\begin{equation}
g^o(\iota ^*(S))\leq deg(S)g^*(\omega _Y^p),
\label{Equation2}\end{equation}
in $A^p(\Gamma )$. 

We can find a curve of automorphisms $\tau (t)\in Aut(\mathbb{P}^N_K)$ for $t\in \mathbb{P}^1_K$ such that for a dense Zariski open dense subset  $U\subset \mathbb{P}^1$, $\tau (t)S$ intersects properly $Y$ and all the irreducible components of $V_l$ (for $l>dim (X)-dim(Y)$)  for all $t\in U$. Let $\mathcal{S}\subset Y\times \mathbb{P}^1$ be the corresponding variety, hence for $t\in U\subset \mathbb{P}^1$, $\mathcal{S}_t=\iota ^*(\tau (t)S)\subset Y$. Since $S$ intersects $Y$ properly, we have $\mathcal{S}_0=\iota ^*(S)$. By the choice of $\mathcal{S}$, for any $t\in U$ the pullback $g^*(\mathcal{S}_t)$ is well-defined as a subvariety of $\Gamma$. 

We consider the induced map $G:\Gamma \times \mathbb{P}^1\rightarrow Y\times \mathbb{P}^1$ given by the formula 
\begin{eqnarray*}
G(z,t)=(g(z),t).
\end{eqnarray*}
We define by $G_0$ the restriction map $G_0:\Gamma \times U\rightarrow Y\times U$. By the choice of the variety $\mathcal{S}$, the inverse image $$G_0^{-1}(\mathcal{S})=G^{-1}(\mathcal{S})\cap  (\Gamma \times U)\subset \Gamma \times U$$ is a closed subvariety of codimension $p$, hence its closure $G^o(\mathcal{S})\subset \Gamma \times \mathbb{P}^1$ is a subvariety of codimension $p$. Moreover, for all $t\in U$ we have
\begin{eqnarray*}
G^o(\mathcal{S})_t=g^*(\mathcal{S}_t).
\end{eqnarray*}
Since the map $g_0:\Gamma -g^{-1}(\mathcal{C}_g)\rightarrow Y-\mathcal{C}_g$ has all fibers of the correct dimension $dim(X)-dim(Y)$, it follows that $$G^o(\mathcal{S})_0\cap (\Gamma -g^{-1}(\mathcal{C}_g))=g_0^{-1}(\iota ^*(S)).$$ 
In fact, let $G_1$ be the restriction of $G$ to $(\Gamma -\mathcal{C}_g)\times \mathbb{P}^1$. Then $$G_1^{-1}(\mathcal{S} )=G^{-1}(\mathcal{S})\cap [(\Gamma -\mathcal{C}_g)\times \mathbb{P}^1]\subset (\Gamma -\mathcal{C}_g)\times \mathbb{P}^1$$
is a closed subvariety of codimension $p$. Hence its closure, denoted by $\widetilde{G}^o(\mathcal{S})\subset \Gamma \times \mathbb{P}^1$ is a subvariety of codimension $p$. For $t\in U$, we have $\widetilde{G}^o(\mathcal{S})_t=g^*(\mathcal{S}_t)=G^o(\mathcal{S}_t)$, because on the one hand $\widetilde{G}^o(\mathcal{S})_t\subset G^{-1}(\mathcal{S})_t=g^*(\mathcal{S}_t)$, and on the other hand $g^*(\mathcal{S}_t)$ has no component on $g^{-1}(\mathcal{C}_g)$ and $\widetilde{G}^o(\mathcal{S})_t\cap (\Gamma -g^{-1}(\mathcal{C}_g))=g_0^{-1}(\mathcal{S}_t)$. Therefore $\widetilde{G}^o(\mathcal{S})=G^o(\mathcal{S})$ as varieties on $\Gamma \times \mathbb{P}^1$. In particular  
\begin{eqnarray*}
G^o(\mathcal{S})_0\cap (\Gamma -g^{-1}(\mathcal{C}_g))=\widetilde{G}^o(\mathcal{S})_0\cap (\Gamma -g^{-1}(\mathcal{C}_g))=g_0^{-1}(\iota ^*(S)),
\end{eqnarray*}
as claimed.

Hence 
\begin{eqnarray*}
g^o(\iota ^*(S))\leq G^o(\mathcal{S})_0
\end{eqnarray*}
as varieties on $\Gamma$. Since $G^o(\mathcal{S})_0$ is rationally equivalent to $G^o(\mathcal{S})_t$ for any $t$ in $U$, it follows that for all such $t$ we have
\begin{eqnarray*}
g^o(\iota ^*(S))\leq G^o(\mathcal{S})_t=g^*(\mathcal{S}_t)=deg(S)g^*(\omega _Y^p),
\end{eqnarray*}
in $A^p(\Gamma )$. Hence (\ref{Equation2}) is proved. 

Now we continue the proof of the lemma. By (\ref{Equation2}) and the bound on degrees (\ref{Equation0}), for all $i=1,\ldots ,e$
\begin{eqnarray*}
g^o(\iota ^*(C_{L_i}(W_{i-1})))\leq Cdeg(W)g^*(\omega _Y^p),
\end{eqnarray*}
in $A^p(\Gamma )$ where $C>0$ is independent of $W$, $X$ and $f$. 

It remains to estimate $g^o(W_e)$. By the choice of $W_e$, the pullback $g^*(W_e)$ is well-defined as a subvariety of $\Gamma$, hence by b) of Lemma \ref{LemmaDegreeOfPullback} and the bound on degrees (\ref{Equation0}) we have
\begin{eqnarray*} 
g^o(W_e)\leq g^*(W_e)\leq Cdeg(W)g^*(\omega _Y^p),
\end{eqnarray*}
in $A^p(\Gamma )$, where $C>0$ is independent of $W$, $X$ and $f$. Thus the proof of the lemma is completed. 
\end{proof} 
  
\subsubsection{Dynamical degrees and some of their basic properties}
We define here dynamical degrees and prove some of their basic properties. When $K$ is the field of complex numbers, all of the results in this subsubsection were known.  Note that in this case (i.e. when $K=\mathbb{C}$), our approach here using Chow's moving lemma is different from the previous ones using "regularization of currents" (see \cite{russakovskii-shiffman} for the case $X=\mathbb{P}^N_{\mathbb{C}}$ the complex projective space and see \cite{dinh-sibony1}\cite{dinh-sibony10} for the case $X$ is a general compact K\"ahler manifold; see also \cite{guedj4}\cite{guedj5} and \cite{friedland}). Let $X$ be a projective manifold with a given embedding $\iota :X\subset \mathbb{P}^N_K$. We let $H\subset \mathbb{P}^N$ be a linear hyperplane, and let $\omega _X=H|_{X}$.

\begin{lemma}
Let $Y,Z$ be projective manifolds, and let $f:Y\rightarrow X,~g:Z\rightarrow Y$ be dominant rational maps. We fix an embedding $Y\subset \mathbb{P}^M_K$ and let $\omega _Y$ be the pullback to $Y$ of a generic hyperplane in $\mathbb{P}^M_K$. Then in $A^p(Z)$ 
\begin{eqnarray*} 
(f\circ g)^*(\omega _X^p) \leq Cdeg(f^*(\omega _X^p))g^*(\omega _Y^p),
\end{eqnarray*}
where $C>0$ is independent of $f$ and $g$.
\label{LemmaDegreeOfCompositionMaps}\end{lemma}
\begin{proof}
We can find proper subvarieties $V_X\subset X,V_Y\subset Y$ and $V_Z\subset Z$ so that the maps $f_0:Y-V_Y\rightarrow X-V_X$ and $g_0:Z-V_Z\rightarrow Y-V_Y$ are regular, proper and have all fibers of the correct dimensions (we can do this by choosing resolutions of singularities for the graphs of $f$ and $g$, and then proceed similarly to the the proof of Lemma \ref{LemmaBoundForPullbackKahlerForms}). Define by $(f\circ g)_0$ the restriction of $f\circ g$ to $Z-V_Z$. Then $(f\circ g)_0=f_0\circ g_0:Z-V_Z\rightarrow X-V_X$ and has all fibers of the correct dimension. We define the strict transforms $f^0$, $g^0$ and $(f\circ g)^0$ using these restriction maps $f_0,g_0$ and $(f\circ g)_0$. 

By Lemma \ref{LemmaDegreeOfPullback} a), we can find a linear subspace $H^p\subset \mathbb{P}^N_K$ so that $H^p$ intersects $X$ properly, $(f\circ g)^*(\iota ^*(H^p))$ is well-defined as a variety and has no component on $V_Z$, and $f^*(\iota ^*(H^p))$ is well-defined as a variety. Then
\begin{eqnarray*}
(f\circ g)^*(\iota ^*(H^p))=(f\circ g)^o(\iota ^*(H^p))
\end{eqnarray*} 
is the closure of $$(f\circ g)_0^{-1}(\iota ^*(H^p))=(f_0\circ g_0)^{-1}(\iota ^*(H^p))=(g_0)^{-1}f_0^{-1}(\iota ^*(H^p)).$$
Therefore $$(f\circ g)^*(\iota ^*(H^p))=g^of^o(\iota ^*(H^p))\leq g^of^*(\iota ^*(H^p)).$$
as subvarieties of $Z$. By Lemma \ref{LemmaDegreeOfStrictTransform}, we have the desired result. 
\end{proof}

Let $f:X\rightarrow X$ be a dominant rational map. Fix a number $p=0,\ldots ,k=dim(X)$. Apply Lemma \ref{LemmaDegreeOfCompositionMaps} to $Y=Z=X$ and the maps $f^n,f^m$, we see that the sequence $n\mapsto deg((f^n)^*(\omega _X^p))$ is sub-multiplicative. Therefore, we can define the $p$-th dynamical degree as follows
\begin{eqnarray*}
\lambda _p(f)=\lim _{n\rightarrow\infty}(deg((f^n)^*(\omega _X^p)))^{1/n}=\inf _{n\in \mathbb{N}}(deg((f^n)^*(\omega _X^p)))^{1/n}.
\end{eqnarray*}
We now relate $\lambda _p(f)$ to the spectral radii $r_p(f^n)$ of the linear maps $(f^n)^*:N_{\mathbb{R}}^p(X)\rightarrow N_{\mathbb{R}}^p(X)$.

\begin{lemma}

a) There is a constant $C>0$ independent of $f$ so that
\begin{eqnarray*}
||f^*(v)||_1\leq C||v||_1||f^*(\omega _X^p)||_1,
\end{eqnarray*}
for all $v\in N_{\mathbb{R}}^p(X)$. Here the norm $||\cdot ||_1$ is defined in (\ref{LabelDefinitionNorm}). Therefore if we denote by $f_p^*$ the linear map $f^*:N^p_{\mathbb{R}}(X)\rightarrow N^p_{\mathbb{R}}(X)$, and by $$||A||_1=\sup _{v\in N^p_{\mathbb{R}}(X), ||v||_1=1}||A(v)||_1$$ the norm of a linear map $A:N^p_{\mathbb{R}}(X)\rightarrow N^p_{\mathbb{R}}(X)$ then
\begin{eqnarray*}
\frac{1}{deg(\omega _X^p)}||f^*(\omega _X^p)||_1\leq ||f_p^*||_1\leq C||f^*(\omega _X^p)||_1,
\end{eqnarray*}
here $C$ is the same constant as in the previous inequality.

b) There is a constant $C>0$ independent of $f$ so that $r_p(f)\leq C||f^*(\omega _X^p)||_1.$

c) We have $\lambda _p(f)=\lim _{n\rightarrow\infty}||(f^n)_p^*||_1^{1/n}\geq\limsup _{n\rightarrow\infty}(r_p(f^n))^{1/n}.$
\label{LemmaSpectralRadius}\end{lemma} 
\begin{proof}

a) Let $m=dim_{\mathbb{Z}}N^p(X)$, and we choose varieties $v_1,\ldots ,v_m$ to be a basis for $N^p(X)$. Then $v_1,\ldots ,v_m$ is also a basis for $N_{\mathbb{R}}^p(X)$. We denote by $||\cdot ||_2$ the max norm on $N_{\mathbb{R}}^p(X)$ with respect to the basis $v_1,\ldots ,v_m$, thus for $v=a_1v_1+\ldots a_mv_m$ 
\begin{eqnarray*} 
||v||_2=\max \{|a_1|,\ldots ,|a_m|\}.
\end{eqnarray*}
By Lemma \ref{LemmaDegreeOfPullback}, we can write each $f^*(v_j)$ as a difference $\alpha _j-\beta _j$ where $\alpha _j$ and $\beta _j$ are effective and $deg(\alpha _j),deg(\beta _j)\leq Cdeg(v_j)deg(f^*(\omega _X^p))$. Here $C>0$ is independent of the map $f$. In particular, $||f^*v_j||_1\leq Cdeg(v_j)deg(f^*(\omega _X^p))$ for any $j=1,\ldots ,m$. Therefore
\begin{eqnarray*}
||f^*v||_1&=&||a_1f^*(v_1)+\ldots a_mf^*(v_m)||_1\leq |a_1||f^*(v_1)|||+\ldots |a_m|||f^*(v_m)||_1\\
&\leq&C||v||_2||f^*(\omega _X^p)||_1\leq C'||v||_1||f^*(\omega _X^p)||_1
\end{eqnarray*}
since any norm on $N_{\mathbb{R}}^p(X)$ is equivalent to $||.||_1$. The other inequalities follow easily from definition of $||f_p^*||_1$. Hence a) is proved.
 
b) Iterating a) we obtain
\begin{eqnarray*}
||(f^*)^nv||_1\leq C^{n-1}||v||_1||f^*(\omega _X^p)||_1^n,
\end{eqnarray*}
for all $n\in \mathbb{N}$ and $v\in N_{\mathbb{R}}^p(X)$. Taking supremum on all $v$ with $||v||_1=1$ and then taking the limit of the $n$-th roots when $n\rightarrow \infty$, we obtain
\begin{eqnarray*}
r_p(f)\leq C ||f^*(\omega _X^p)||_1.
\end{eqnarray*}
Here $C>0$ is the same constant as in a).

c) follows from a) and b) and the definition of the dynamical degree $\lambda _p(f)$. 

\end{proof}

Using Lemma \ref{LemmaDegreeOfCompositionMaps}, it is standard (see e.g. \cite{dinh-sibony10}) to prove the following result
\begin{lemma} The dynamical degrees are birational invariants. More precisely, if $X,Y$ are projective manifolds of the same dimension $k$, $f:X\rightarrow X$ and $g:Y\rightarrow Y$ are dominant rational maps, and $\pi :X\rightarrow Y$ is a birational map so that $\pi \circ f=g\circ \pi$, then $\lambda _p(f)=\lambda _p(g)$ for all $p=0,\ldots ,k$. 
\label{LemmaDynamicalDegreeBirationalInvariant}\end{lemma}

It is also possible to prove the log-concavity of dynamical degrees in the algebraic setting, provided that a mixed Hodge-Riemann theorem is valid for the algebraic setting (for mixed Hodge-Riemann theorem on compact K\"ahler manifolds, see the paper Dinh-Nguyen \cite{dinh-nguyen}). Nevertheless, we have the following direct consequence of Lemma \ref{LemmaBoundForPullbackKahlerForms}
\begin{lemma}
Let $f:X\rightarrow X$ be a rational map. For any $p=0,\ldots ,k-1$
\begin{eqnarray*}
\lambda _1(f)\lambda _p(f)\geq \lambda _{p+1}(f).
\end{eqnarray*}
In particular, $\lambda _1(f)^p\geq \lambda _p(f)$ for any $p=0,\ldots ,k$.
\label{LemmaBoundForDynamicalDegrees}\end{lemma}  

\subsubsection{$p$-stability}
Let $X$ be a projective manifold, and $f:X\rightarrow X$ a dominant rational map. Given $p=0,\ldots ,k=dim (X)$, we say that $f$ is $p$-stable if for any $n\in \mathbb{N}$, $(f^n)^*=(f^*)^n$ on $N^{p}_{\mathbb{R}}(X)$. Note that when $K=\mathbb{C}$ and $p=1$ or $p=k-1$ this is the usual definition. In fact, Lefschetz theorem on $(1,1)$ classes and the hard Lefschetz theorem (see e.g. Chapter 0 in \cite{griffiths-harris}) imply that if $X$ is a complex projective manifold then $H^{1,1}(X)$ and $H^{k-1,k-1}(X)$ are generated by algebraic cycles.

\subsection{Analogs of Theorems \ref{TheoremFirstDynamicalDegreeIsSimple} and \ref{TheoremTheCaseNotStable}}   

First, we state the analogs of Lemma \ref{LemmaPullPushFormulaForOneBlowup} and Proposition \ref{LemmaPullPushFormulaForMeromorphicMaps}.

\begin{lemma} Let $X\subset \mathbb{P}^N_K$ be a projective manifold of dimension $k$. Let $\pi :Z\rightarrow X$ be a blowup of $X$ along a smooth  submanifold $W=\pi (E)$ of codimension exactly $2$. Let $E$ be the exceptional divisor and let $L$ be a general fiber of $\pi$. Let $\alpha$ be a vector in $N^1_{\mathbb{C}}(Z)$. 

i) In $A^*(X)$ we have
\begin{eqnarray*}
(\pi )_*(E.E)=-W.
\end{eqnarray*}

ii)  
\begin{eqnarray*}
\pi ^*(\pi )_*(\alpha )=\alpha +(\{\alpha \}.\{L\})E.
\end{eqnarray*}

iii)  
\begin{eqnarray*}
(\pi )_*(\alpha . E)=(\{\alpha \}.\{L\})W.
\end{eqnarray*}

iv) 
\begin{eqnarray*}
(\pi )_*(\alpha ).(\pi )_* (\overline{\alpha })-(\pi )_*(\alpha . \overline{\alpha })=|\{\alpha \}.\{L\}|^2W.
\end{eqnarray*}
\label{LemmaPullPushFormulaForOneBlowupAlgebraicCase}\end{lemma}
\begin{proof}
i) follows from the formula at the beginning of Section 4.3 in \cite{fulton}. Then ii), iii) and iv) follows from i) as in the proof of Lemma \ref{LemmaPullPushFormulaForOneBlowup}.
\end{proof}

\begin{proposition}
Let $X$ and $Y$ be projective manifolds, and $h:X\rightarrow Y$ a dominant rational map. Further, let $u\in N^1_{\mathbb{C}}(Y)$, then $h^*(u).h^*(\overline{u} )-h^*(u . \overline{u})\in N^2_{\mathbb{R}}(X)$ is effective.
\label{LemmaPullPushFormulaForMeromorphicMapsAlgebraicCase}\end{proposition}
\begin{proof}
The proof is identical with that of Proposition \ref{LemmaPullPushFormulaForMeromorphicMaps}, the only difference here is that we use Hironaka's elimination of indeterminacies for rational maps on projective manifolds over algebraic closed fields of characteristic zero (see e.g. Corollary 1.76 in Koll\'ar \cite{kollar} and Theorem 7.21 in Harris \cite{harris}). (In the algebraic case, the Hironaka's elimination of indeterminacies for a rational map $f:X->Y$ is a consequence of the basic monomialization theorem, applied to the ideal generated by the components of the map $f$ in an ambient projective space of $Y$. We thank )
\end{proof}

Now we state the analogs of Theorems \ref{TheoremFirstDynamicalDegreeIsSimple} and \ref{TheoremTheCaseNotStable}. We omit the proofs of these results here since they  are similar to those of Theorems \ref{TheoremFirstDynamicalDegreeIsSimple} and \ref{TheoremTheCaseNotStable}. 

\begin{theorem} Let $X\subset \mathbb{P}^N_K$ be a projective manifold of dimension $k$, and let $f:X\rightarrow X$ be a dominant rational map which is $1$-stable. Assume that $\lambda _1(f)^2>\lambda _2(f)$. Then $\lambda _1(f)$ is a simple eigenvalue of $f^*:N^1_{\mathbb{R}}(X)\rightarrow N^1_{\mathbb{R}}(X)$. Further, $\lambda _1(f)$ is the only eigenvalue of modulus greater than $\sqrt{\lambda _2(f)}$.  
\label{TheoremFirstDynamicalDegreeIsSimpleAlgebraicCase}\end{theorem}

\begin{theorem}
Let $X\subset \mathbb{P}^N_K$ be a projective manifold, and let $f:X\rightarrow X$ be a dominant rational map. Assume that $f^*:N^2_{\mathbb{R}}(X)\rightarrow N^2_{\mathbb{R}}(X)$ preserves the cone of effective classes. Then  

1) We have $r_1(f)^2\geq r_2(f)$.

2) Assume moreover that $r_1(f)^2>r_2(f)$. Then $r_1(f)$ is a simple eigenvalue of $f^*:N^1_{\mathbb{R}}(X)\rightarrow N^1_{\mathbb{R}}(X)$. Further, $r_1(f)$ is the only eigenvalue of modulus greater than $\sqrt{r_2(f)}$.  

\label{TheoremTheCaseNotStableAlgebraicCase}\end{theorem}

\end{document}